\documentclass[11pt]{amsart}

\usepackage[utf8]{inputenc}
\usepackage[english]{babel}
\usepackage{amsmath,amsfonts,amssymb}
\usepackage{mathtools}
\usepackage{mathrsfs}
\usepackage[colorlinks=true]{hyperref}
\usepackage{csquotes}
\usepackage{paralist}

\usepackage{tikz}
\usetikzlibrary{cd,arrows}
\usepackage{pgfplots}
\pgfplotsset{width=12cm,compat=1.9}

\usepackage[backend=biber,style=alphabetic,doi=false,isbn=false,url=true,datamodel=eprint-hal]{biblatex}
\DeclareFieldFormat{hal}{%
  HAL\addcolon\space
  \ifhyperref
    {\href{https://hal.archives-ouvertes.fr/#1}{\nolinkurl{#1}}}
    {\nolinkurl{#1}}}
    
\DeclareFieldAlias{eprint:hal}{hal}
\DeclareFieldAlias{eprint:HAL}{eprint:hal}

\renewbibmacro*{eprint}{%
  \printfield{hal}%
  \newunit\newblock
  \iffieldundef{eprinttype}
    {\printfield{eprint}}
    {\printfield[eprint:\strfield{eprinttype}]{eprint}}}

\theoremstyle{plain}
\newtheorem{thm}{Theorem}[section]
\newtheorem{prp}[thm]{Proposition}
\newtheorem{cor}[thm]{Corollary}
\newtheorem{lem}[thm]{Lemma}

\theoremstyle{definition}
\newtheorem{dfn}[thm]{Definition}

\theoremstyle{remark}
\newtheorem{rmk}[thm]{Remark}
\newtheorem{exa}[thm]{Example}

\newtheorem{prb}[thm]{Problem}

\newcommand{\CC}{\mathbb{C}}
\newcommand{\NN}{\mathbb{N}}
\newcommand{\QQ}{\mathbb{Q}}
\newcommand{\RR}{\mathbb{R}}
\newcommand{\ZZ}{\mathbb{Z}}

\renewcommand{\O}{\mathcal{O}}
\newcommand{\I}{\mathcal{I}}
\newcommand{\J}{\mathcal{J}}
\newcommand{\X}{\mathcal{X}}

\newcommand{\M}{\mathsf{M}}

\newcommand{\mm}{\mathfrak{m}}

\newcommand{\be}{\mathbf{e}}
\newcommand{\bt}{\mathbf{t}}
\newcommand{\bw}{\mathbf{w}}
\newcommand{\bx}{\mathbf{x}}

\newcommand{\Cx}{\CC\{\bx\}}

\DeclarePairedDelimiter\set\{\}
\DeclarePairedDelimiter\abs\lvert\rvert 
\DeclarePairedDelimiter\norm\lVert\rVert
\DeclarePairedDelimiter\ideal\langle\rangle

\DeclareMathOperator{\ann}{ann}
\DeclareMathOperator{\Aut}{Aut}
\DeclareMathOperator{\Der}{Der}
\DeclareMathOperator{\edim}{edim}
\DeclareMathOperator{\Iso}{Iso}
\DeclareMathOperator{\Min}{Min}
\DeclareMathOperator{\red}{red}
\DeclareMathOperator{\rk}{rk}
\DeclareMathOperator{\Sing}{Sing}
\DeclareMathOperator{\supp}{supp}
\DeclareMathOperator{\syz}{syz}

\title[Monomial Jacobian ideals]{Hypersurface singularities with\\ monomial Jacobian ideal}

\author[R.~Epure]{Raul Epure}
\address{ \linebreak
Raul Epure\\
Department of Mathematics, TU Kaiserslautern\\
67663 Kaiserslautern\\
Germany
}
\email{\href{epure@mathematik.uni-kl.de}{epure@mathematik.uni-kl.de}}

\author[M.~Schulze]{Mathias Schulze}
\address{ \linebreak
Mathias Schulze\\
Department of Mathematics, TU Kaiserslautern\\
67663 Kaiserslautern\\
Germany
}
\email{\href{mschulze@mathematik.uni-kl.de}{mschulze@mathematik.uni-kl.de}}



\thanks{}

\numberwithin{equation}{section}

\begin{document}

\begin{abstract}
We show that every convergent power series with monomial extended Jacobian ideal is right equivalent to a Thom--Sebastiani polynomial.
This solves a problem posed by Hauser and Schicho.

On the combinatorial side, we introduce a notion of Jacobian semigroup ideal involving a transversal matroid.
For any such ideal we construct a defining Thom--Sebastiani polynomial.

On the analytic side, we show that power series with a quasihomogeneous extended Jacobian ideal are strongly Euler homogeneous.
Due to a Mather--Yau-type theorem, such power series are determined by their Jacobian ideal up to right equivalence.
\end{abstract}

\maketitle

\section{Introduction}\label{23}

Let $f\colon U\to\CC$ be a holomorphic function, defined in some open neighborhood $U\subseteq\CC^n$ of $\mathbf0\in\CC^n$.
By considering arbitrarily small $U$, $f$ can be considered as a convergent power series $f\in\Cx$ in variables $\bx=x_1,\dots,x_n$, which are local coordinates on $\CC^n$ at $\mathbf0$.
Similarly, a local biholomorphic map of $\CC^n$ at the origin can be seen as a 
$\CC$-algebra automorphism $\varphi\in\Aut_\CC\Cx$, or a local coordinate change of $\bx$.


The \emph{Jacobian ideal} and \emph{extended Jacobian ideal} of $f$ are defined by
\[
\J_f=\ideal{\frac{\partial f}{\partial x_1},\dots,\frac{\partial f}{\partial x_n}}\unlhd\Cx,\quad
\tilde\J_f=\ideal{\frac{\partial f}{\partial x_1},\dots,\frac{\partial f}{\partial x_n},f}\unlhd\Cx,
\]
respectively.


\begin{rmk}\label{32}
The ideals $\J_f$ and $\tilde\J_f$ are analytic invariants of $f\in\Cx$ and of the ideal $\ideal f\unlhd\Cx$, respectively.
This means that if $\varphi\in\Aut_\CC\Cx$ is a $\CC$-algebra automorphism and $u\in\Cx^*$ a unit power series, then 
\[
\varphi(\J_f)=\J_{\varphi(f)},\quad\tilde\J_{u\cdot f}=\tilde\J_f.
\]
In particular, $\varphi$ induces $\CC$-algebra isomorphisms
\[
\Cx/\J_f\to\Cx/\J_{\varphi(f)},\quad\Cx/\tilde\J_f\to\Cx/\tilde\J_{\varphi(f)}.
\]
However, $\J_{u\cdot f}\ne\J_f$, for instance, for $f=x^5+x^2y^2+x^5$ and $u=1+x$.
\end{rmk}


In the language of analytic geometry, the arbitrarily small neighborhoods $U\subseteq\CC^n$ of $\mathbf0\in\CC^n$ with functions $f\in\Cx$ on it form a \emph{smooth space germ} 
\[
Y=(\CC^n,\mathbf0),\quad\O_Y=\Cx\rhd\ideal\bx=\mm_Y.
\]
A \emph{subspace germ} $X=V(\I_{X/Y})\subseteq Y$ is the zero locus of an ideal $\I_{X/Y}\unlhd\O_Y$, equipped with the $\CC$-algebra $\O_X=\O_Y/\I_{X/Y}$.
If $\I_{X/Y}$ is a radical ideal, that is, $\O_X$ and $X$ are \emph{reduced}, this algebraic structure is redundant.
If $\I_{X/Y}=\ideal f$ is generated by a single power series $f\in\Cx$, then
\[
\emptyset\ne V(f)\iff\mathbf0\in V(f)\iff f\in\mm_Y,\quad V(f)\subsetneq Y\iff f\ne0.
\]
If both latter conditions hold true, then the subspace germ
\[
\mathbf0\in X=V(f)\subsetneq Y,\quad \O_X=\O_Y/\ideal f\rhd\mm_Y/\ideal{f}=\mm_X,
\]
is called a \emph{hypersurface singularity}.
This means that $X$ is of (pure) codimension $1$ in the smooth space germ $Y$.
The points in $Y$ where all partial derivatives of $f$ vanish are called \emph{critical points} of $f$, those in $X$, \emph{singular points} of $X$.
The corresponding subspace germ is the \emph{singular locus} of $X$,
\[
\Sing X=V(\tilde\J_f)\subseteq X,\quad\O_{\Sing X}=\O_Y/\tilde\J_f=\O_X/\O_X\J_f.
\]
If $\Sing X=\emptyset$ is empty, then $f$ is a coordinate and $X\cong(\CC^{n-1},\mathbf0)$ is smooth.
Otherwise, 
\[
\mathbf0\in\Sing X\iff 0\ne f\in\mm_Y^2.
\]
If $\Sing X=\set{\mathbf0}$ is a point, then $f$ has an \emph{isolated critical point} $\mathbf0$, and $X=V(f)$ is called an \emph{isolated hypersurface singularity}.


\begin{rmk}\label{39}
In more intrinsic terms, $n$ equals the \emph{embedding dimension}
\[
\edim X:=\dim_\CC(\mm_X/\mm_X^2)
\]
of the hypersurface singularity $X=V(f)$, unless $\Sing X=\emptyset$ where $\edim X=n-1$.
Furthermore, the \emph{Jacobian ideal} $\O_X\J_f\unlhd\O_X$ of $X$ defining $\Sing X$ is the Fitting ideal of order $\dim X$ of the $\O_X$-module $\Omega_X^1$ of differential $1$-forms on $X$, and $\tilde\J_f$ is its contraction to $\O_Y$.
\end{rmk}


Any isomorphism $X\cong X'=V(f')\subseteq Y$ lifts to an automorphism of $Y$, defined by some $\varphi\in\Aut_\CC\Cx$ with $\varphi(f)=u\cdot f'$ for some unit $u\in\Cx^*$.


\begin{dfn}\label{31}
Two power series $f,f'\in\Cx$ are called \emph{contact equivalent} if $\varphi(f)=u\cdot f'$ for some $\CC$-algebra automorphism $\varphi\in\Aut_\CC\Cx$ and some unit power series $u\in\Cx^*$.
They are called \emph{right equivalent} if $u=1$.
\end{dfn}


Due to Remark~\ref{32}, $X\cong X'$ then implies that $\Sing X\cong\Sing X'$.
The converse implication is a celebrated theorem of Mather and Yau in the case of isolated hypersurface singularities (see \cite{MY82}), and of Gaffney and Hauser for general hypersurface singularities (see \cite[Part~I]{GH85} and \cite[Thm.~2]{HM86}) under the following mild restriction (see \cite[\S3, Def.]{HM86}).


\begin{dfn}\label{40}
The \emph{isosingular locus} of the germ $X=(\X,x)$ of a space $\X$ at $x$ is the subspace germ 
\[
\Iso X=(\set{x'\in\X\mid (\X,x')\cong X},x)\subseteq X.
\]
Note that $\Iso X\subseteq\Iso\Sing X$ if $\Sing X\ne\emptyset$.
If $\Iso X\supseteq\Iso\Sing X$, then $X$ is called \emph{harmonic}, and \emph{dissonant} otherwise.
\end{dfn}


Due to Ephraim (see \cite[Thm.~0.2]{Eph78}), 
\begin{equation}\label{34}
\Iso X \cong(\CC^k,\mathbf0),\quad X\cong X'\times\Iso X,\quad\Sing X\cong\Sing X'\times\Iso X.
\end{equation}
If $k=0$, then we say that $\Iso X$ is \emph{trivial}.


\begin{thm}[Mather--Yau, Gaffney--Hauser]\label{5}
If $X$ and $X'$ are harmonic hypersurface singularities, then $X\cong X'$ if and only if $\dim X=\dim X'$ and $\Sing X\cong\Sing X'$.
\end{thm}


In other words, the extended Jacobian ideal $\tilde\J_f$ determines the geometry of the hypersurface singularity $X=V(f)$.
This stunning fact has not been exploited systematically so far.
It is natural to study hypersurface singularities $X$ where $\Sing X$ is particularly simple.
In this spirit, Hauser and Schicho (see \cite[Prob.~2\textsuperscript\footnotesize{*}]{HS11}) formulated the following


\begin{prb}[Hauser--Schicho]\label{26}
Describe all power series $f\in\CC\{\bx\}$ for which the extended Jacobian ideal $\tilde\J_f$ is a \emph{monomial ideal}, that is, generated by monomials in terms of some local coordinates $\bx$ on $\CC^n$ at $\mathbf0$.
\end{prb}


Obviously this happens if $f$ is of the following type.

\begin{dfn}
We call a nonzero sum of non-constant monomials in disjoint sets of variables a \emph{Thom--Sebastiani polynomial}.
The particular case of a nonzero sum of positive powers of all variables is called a \emph{Brieskorn--Pham polynomial}.
\end{dfn}


\begin{exa}\label{43}
The \emph{Whitney umbrella} $X=V(f)$ defined by the Thom--Sebastiani polynomial $f=x^2+y^2z$ has a monomial extended Jacobian ideal $\tilde\J_f=\ideal{x,y^2,yz}$.
\end{exa}


We record some obvious properties of Thom--Sebastiani polynomials.

\begin{rmk}\label{33}
Let $f\in\CC\{\bx\}$ be a Thom--Sebastiani polynomial.
\begin{enumerate}[(a)]

\item\label{33a} Then $f$ is quasihomogeneous (see Remark~\ref{56}).

\item\label{33b} Unless $f$ is a monomial, it is squarefree.
Indeed, any multiple factor $g$ of $f$ divides a monomial $\frac{\partial f}{\partial x_i}$, which forces $g$ and hence $f$ to be a monomial.

\item\label{33c} For any unit $u\in\CC\{\bx\}^*$, $u\cdot f$ is right equivalent to a Thom--Sebastiani polynomial. 
In fact, for each monomial $\bx^\alpha$ of $f$ and a choice of $i$ with $\alpha_i\ne0$, considering $\sqrt[\alpha_i]u\cdot x_i$ as a new variable eliminates $u$.

\item\label{33d} By right equivalence, all degree two monomials of $f$ can be turned into squares. 
Indeed, a linear coordinate change replaces $x_ix_j$ by $x_i^2+x_j^2$.

\end{enumerate}
\end{rmk}


Our main result solves Problem~\ref{26} of Hauser and Schicho.

\begin{thm}\label{6}
The extended Jacobian ideal $\tilde\J_f$ of a power series $0\ne f\in\ideal\bx\lhd\Cx$ is monomial if and only if $f$ is right equivalent to a Thom--Sebastiani polynomial.
\end{thm}


Its proof in \S\ref{9} relies on a combinatorial study of monomial Jacobian ideals in \S\ref{8}:
By passing to exponents of $f$ we introduce a notion of \emph{Jacobian semigroup ideal} which implements the underlying linear algebra in terms of a transversal matroid (see Definition~\ref{2} and Remark~\ref{13}).
In Proposition~\ref{1}, we show that every such semigroup ideal arises from the exponents of a Thom--Sebastiani polynomial $f'$.
The claimed right equivalence of $f$ and $f'$ then follows from a Mather--Yau-type Theorem~\ref{4} for strongly Euler homogeneous power series (see Definition~\ref{25}).
The homogeneity hypothesis is satisfied if $\tilde\J_f$ is monomial due to Theorem~\ref{3}, which generalizes a result of Xu and Yau in the isolated singularity case (see \cite[Thm.~1.2]{XY96}).

\smallskip


We collect some consequences of Theorem~\ref{6}.
In the isolated singularity case a result of K.~Saito yields

\begin{cor}\label{18}
If $f\in\Cx$ has an isolated critical point and $\tilde\J_f$ is monomial, then $f$ is right equivalent to a Brieskorn--Pham polynomial.
\end{cor}

\begin{proof}
By Theorem~\ref{6}, we may assume that $0\ne f\in\mm_Y^2$ is a Thom--Sebastiani polynomial with isolated critical point.
In particular, $f$ is quasihomogeneous by Remark~\ref{33}.\eqref{33a}.
Then there must be, for each $i\in\set{1,\dots,n}$, a monomial $x_i^m$ or $x_i^mx_j$ in $f$ where $m\ge1$ (see \cite[Kor.~1.6]{Sai71}).
In the second case, switching $i$ and $j$ forces $m=1$ and Remark~\ref{33}.\eqref{33d} applies.
\end{proof}


In geometric terms, $\tilde\J_f$ monomial means that $\Sing X$ is normal crossing (see Definition~\ref{16}).
Using Remarks~\ref{33}.\eqref{33a} and \eqref{33b} and Proposition~\ref{51} we obtain

\begin{cor}\label{17}
Any hypersurface singularity $X=V(f)$ with normal crossing singular locus $\Sing X$ is quasihomogeneous, holonomic and either reduced, or a (possibly non-reduced) normal crossing divisor.\qed
\end{cor}


Due to the Aleksandrov--Terao Theorem, the notion of \emph{Saito-free divisor} generalizes by requiring Cohen--Macaulayness of a generalized Jacobian ideal (see \cite[Def.~5.1]{GS12}, \cite[Def.~5.5]{Sch16}, \cite[Def.~4.3]{Pol20}).
Results of Epure and Pol (see \cite[Cor.~5.5]{Pol20} and \cite[Thm.~2]{EP21}) yield the final conclusion in

\begin{cor}\label{19}
A reduced normal crossing singular locus of a hypersurface singularity is a Cartesian product of equidimensional unions of coordinate subspaces, and hence, a free singularity.\qed
\end{cor}


We conclude with an application to E.~Faber's conjecture, which aims for characterizing non-smooth normal crossing divisors as hypersurface singularities $X=V(f)$ with $\J_f$ radical and equidimensional of height $2$ (see \cite[Conj.~2]{Fab15}).
From $\J_f$ radical it follows that $f$ is Euler homogeneous (see \cite[Lem.~1]{Fab15}).
In particular, $\J_f=\tilde\J_f$ depends only on $X$.

\begin{cor}\label{20}
A non-smooth hypersurface singularity $X=V(f)$ is a reduced normal crossing divisor if and only if $\J_f$ is monomial and radical of height $2$.
\end{cor}

\begin{proof}
Suppose that $\J_f$ is monomial and radical of height $2$.
By Theorem~\ref{6}, $f$ is then a Thom--Sebastiani polynomial of squarefree monomials in terms of variables $\bx=x_1,\dots,x_n$.
Each monomial contributes $1$ to the height of $\J_f$ if quadratic, and $2$ otherwise.
Then either $f=x_1^2+x_2^2$, or $f\in\ideal{\bx}^3$ is a monomial.
In the first case, $f=x_1x_2$ after a linear coordinate change.
Thus, $X$ is a reduced normal crossing divisor in both cases.
\end{proof}

\subsection*{Acknowledgements}

Groundwork for this article was laid in the Ph.D.~thesis of the first-named author (see \cite{Epu20}) under the direction of the second named author.
We thank Herwig Hauser for reviewing this thesis and for many valuable comments and suggestions.

\section{Jacobian semigroup ideals}\label{8}

In this section, we describe the combinatorics underlying our problem by combining the structures of semigroup ideals and transversal matroids.


Fix $n\in\NN$ and set $[n]:=\set{1,\dots,n}$.
Recall that a \emph{matroid} $\M$ on $[n]$ axiomatizes the notion of linear dependence of families of $n$ vectors in a vector space (see \cite{Oxl11}).
Among other options, it can be defined by the data of \emph{independent sets}, or that of \emph{bases}, that is, maximal independent sets.
Both are distinguished subsets of the ground set $[n]$, subject to corresponding matroid axioms.
These latter implement standard theorems of linear algebra such as the Steinitz exchange lemma and the basis extension theorem.
The \emph{rank} $\rk(S)$ of a subset $S\subseteq[n]$ is the maximal cardinality of an independent subset of $S$.
The rank $\rk\M:=\rk([n])$ of the matroid $\M$ equals the cardinality of any basis.


Consider the commutative monoid $M:=(\NN^n,+)$.
The \emph{support} of an element $\alpha\in M$ is the set 
\begin{equation}\label{55}
[\alpha]:=\set{i\in[n]\mid\alpha_i\ne0},
\end{equation}
its \emph{degree} is defined by
\[
\abs{\alpha}:=\alpha_1+\dots+\alpha_n.
\] 
For any subset $F\subseteq M$, we consider the union of supports of all its elements,
\begin{equation}\label{57}
[F]:=\bigcup_{\alpha\in F}[\alpha].
\end{equation}


For any semigroup ideal $I\unlhd M$ the set of minima $\Min(I)$ with respect to the partial ordering is the (unique) minimal set of generators.
By Dickson's Lemma, it has finite cardinality
\[
\mu(I):=\abs{\Min(I)}<\infty.
\]


For $i\in[n]$, denote by $\be_i:=(\delta_{i,j})_j\in M$ the $i$th unit vector and define \emph{partial differentiation operators} 
\[
\delta_i:M\to M\cup\set{\infty},\quad
\delta_i(\alpha):=\begin{cases}
\infty & \text{if } \alpha_i=0,\\
\alpha-\be_i & \text{otherwise}.
\end{cases}
\]
For $F\subseteq M\ni\alpha$, we write 
\[
\delta(F):=\bigcup_{i\in[n]}\delta_i(F)\setminus\set\infty\subseteq M,\quad
\delta(\alpha):=\delta(\set\alpha).
\]


\begin{dfn}\label{2}
Let $F\subseteq M$ be a subset and consider the semigroup ideal
\[
J_F:=\ideal{\delta(F)}\unlhd M.
\]
We call the \emph{transversal matroid} $\M_F$ associated with the covering of $\Min(J_F)$ 
\[
\set{\Min(J_F)\cap\delta_i(F)\mid i\in[n]}\subseteq 2^{\Min(J_F)}
\]
the \emph{Jacobian matroid} of $F$ (see \cite[\S1.6]{Oxl11}).
Its independent sets are the \emph{partial transversals} of the covering, that is, injective maps
\[
\begin{tikzcd}
\left[n\right]\supseteq I\ar[hookrightarrow]{r}{\psi} & \Min(J_F),\quad\text{where }\psi(i)\in\delta_i(F)\text{ for all }i\in I.
\end{tikzcd}
\]
Note that the \emph{rank} of $\psi$ equals $\rk(\psi)=\abs I$.
If $\rk\M_F=\mu(J_F)$, then refer to $J_F$ as the \emph{Jacobian semigroup ideal} of $F$.
\end{dfn}


Note that it is a strong requirement on $F$ to have a Jacobian semigroup ideal $J_F$.
In fact, typically $\abs{\mu(J_F)}>n$, whereas $\rk\M_F\le n$.


Our terminology is motivated by the following

\begin{rmk}\label{13}
Consider the \emph{support} of a power series in variables $\bx=x_1,\dots,x_n$,
\[
f=\sum_{\alpha}f_\alpha\bx^\alpha\in\CC\{\bx\},\quad F:=\supp(f)=\set{\alpha\in M\mid f_\alpha\ne0}\subseteq M.
\]
Suppose that $\J_f$ is generated by monomials in terms of the variables $\bx$.
Then $J_F$ is the set of exponents of monomials in $\J_f$, and $\Min(J_F)$ is the subset of exponents of minimal monomial generators of $\J_f$.
Since $\J_f$ is generated by $n$ many partial derivatives, its minimal number of generators is bounded by
\[
\mu(J_F)=\mu(\J_f)\le n.
\]
By Nakayama's Lemma, any minimal generators $\frac{\partial f}{\partial x_i}$, $i\in I$, of $\J_f$ map to a basis of the $\CC$-vector space $\J_f/\ideal{\bx}\J_f$, which has a monomial basis with exponents in $\Min(J_F)$.
Gaussian Elimination yields a bijection $\psi\colon I\to\Min(J_F)$ such that $\bx^{\psi(i)}$ is a monomial of $\frac{\partial f}{\partial x_i}$, and hence, $\psi(i)\in\delta_i(F)$ for all $i\in I$.
Thus, $\psi$ is a partial transversal.
It follows that
\[
\mu(\J_f)=\abs{I}\le\rk\M_F\le\mu(J_F)=\mu(\J_f)
\]
is an equality.
This makes $J_F$ a Jacobian semigroup ideal.
\end{rmk}


The impression that Jacobian semigroup ideals are quite special is further supported by our main combinatorial result.
It replaces $F$ by the support of a Thom--Sebastiani polynomial leaving $J_F$ unchanged.
We first illustrate its proof in the simple case of the Whitney umbrella from Example~\ref{43}.


\begin{exa}\label{42}
For $f=x^2+y^2z$ with $J_f=\langle x,yz,y^2\rangle$, we obtain
\begin{align*}
F&=\set{(2,0,0),(0,2,1)},\\
\delta_1(F)&=\set{(1,0,0)},\quad\delta_2(F)=\set{(0,1,1)},\quad\delta_3(F)=\set{(0,2,0)},\\
\delta((2,0,0))&=\set{(1,0,0)},\quad\delta((0,2,1))=\set{(0,1,1),(0,2,0)},\\
\Min(J_F)&=\set{(1,0,0),\ (0,1,1),\ (0,2,0)}.
\end{align*} 
We can reconstruct $F$ from $J_F$ as follows:
Choose
\[
\psi(1):=(1,0,0)\in\Min(J_F),\quad\psi(1)=\delta_1(\alpha),\quad\alpha:=(2,0,0)\in F,
\]
to obtain a partial transversal $\psi\colon[1]\hookrightarrow\Min(J_F)$.
Set $F':=\set{\alpha}$.
Since
\[
\rk(\psi)=\abs{[1]}=1<3=\mu(J_F)=\rk\M_F,
\]
$\psi$ extends to $[2]$ by
\[
\psi(2)\in\delta_2(F)=\set{\delta_2((0,2,1))},\quad(0,2,1):=\alpha',
\]
and $F'\cup\set{\alpha'}=F$. 
Then 
\[
\delta(F')=\psi([2])\sqcup\set{\delta_3(\alpha')},\quad\delta_3(\alpha')=:\psi(3),
\]
extends $\psi$ to $[3]$ and the process terminates.
\end{exa}


We now develop the approach of Example~\ref{42} into a general argument.


\begin{prp}\label{1}
Let $F\subseteq M$ be a subset such that $J_F$ is a Jacobian semigroup ideal.
Then $J_F=J_{F'}$ for some subset $F'\subseteq F\setminus\set{\mathbf0}$ whose elements $\alpha\in F'$ have disjoint supports $[\alpha]$ and contribute minimal generators $\delta(\alpha)\subseteq\Min(J_F)$ of $J_F$.
\end{prp}

\begin{proof}
For increasing $\ell$, we construct partial transversals
\[
\begin{tikzcd}
\left[\ell\right]\ar[hookrightarrow]{r}{\psi} & \Min(J_F),
\end{tikzcd}
\]
enumerating $\Min(J_F)$ by increasing degree, together with subsets $F'\subseteq F$ such that
\[
[F']=\bigsqcup_{\alpha\in F'}[\alpha]=[\ell],\quad \delta(F')=\psi([\ell]).
\]
The claim is proven when equality has been reached in
\[
\rk(\psi)=\ell\le\mu(J_F)=\rk\M_F.
\]
Otherwise, an $\alpha\in F\setminus\set{\mathbf0}$ extends $\psi$ to a $k\in[\alpha]\setminus[\ell]$ by a new minimal generator
\[
\beta:=\psi(k)=\delta_k(\alpha)\in\Min(J_F)\setminus\psi([\ell]).
\]
Since $\psi$ is part of a basis of $\M_F$, the degree $\abs{\beta}$ can be chosen minimal.
Suppose that, for some $i\in[\alpha]$, $\delta_i(\alpha)\not\in\Min(J_F)\setminus\psi([\ell])$ is not a new minimal generator.
Since $\psi$ enumerates $\Min(J_F)$ by increasing degree, this means that $\delta_i(\alpha)\in\ideal{\psi([\ell])}$.
Then there is a $j\in[\ell]$ such that
\[
\alpha-\be_i=\delta_i(\alpha)\ge\psi(j)=:\gamma.
\]
Using that 
\[
[\alpha]\ni k\not\in[\ell]=[F']\supseteq[\delta(F')]=[\psi([\ell])]\supseteq[\gamma]\implies\alpha_k>\gamma_k,
\]
this leads to the contradiction
\[
\Min(J_F)\ni\beta>\beta-\be_i=\alpha-\be_i-\be_{k}\ge\gamma\in J_F.
\]
It follows that all $\abs{[\alpha]}$ elements of $\delta(\alpha)\subseteq\Min(J_F)\setminus\psi([\ell])$ are pairwise different new minimal generators of $J_F$ of minimal degree $\abs\alpha-1=\abs\beta$.
In particular, $\psi$ extends to $[\ell]\cup[\alpha]$.
Suppose that $i\in[\ell]\cap[\alpha]$, and hence,
\[
\rk(\psi)=\abs{[\ell]\cup[\alpha]}<\ell+\abs{[\alpha]}\le\mu(J_F)=\rk\M_F.
\]
An extension of $\psi$ yields a $j\not\in[\ell]\cup[\alpha]$ and an $\alpha'\in F$ such that $\delta_i(\alpha)=\delta_j(\alpha')$.
Replace $\alpha$ by $\alpha'$ repeatedly to decrease $\abs{\alpha\vert_{[\ell]}}$ until $[\ell]\cap[\alpha]=\emptyset$.
Then reorder $\ell+1,\dots,n$ such that $[\ell]\cup[\alpha]=[\ell+\abs{[\alpha]}]$.
Now including $\alpha$ in $F'$ increases $\ell$ by $\abs{[\alpha]}$.
Iterating this procedure until $\ell=\mu(J_F)$ yields the claim.
\end{proof}

\section{Mather--Yau under strong Euler homogeneity}

In this section, we discuss a Mather--Yau-type theorem for strongly Euler homogeneous power series.
We first recall the definition (see \cite[p.~769]{GS06}).


\begin{dfn}\label{25}
A power series $f\in\O_Y=\Cx$ is called \emph{(strongly) Euler homogeneous} if $f\in\J_f$ ($f\in\mm_Y\J_f$ where $\mm_Y=\ideal\bx$).
In this case, a hypersurface singularity $X=V(f)\subseteq Y$ is called \emph{(strongly) Euler homogeneous}.
\end{dfn}


\begin{rmk}\label{21}\
\begin{enumerate}[(a)]

\item\label{21a} Euler homogeneity of $f$ is equivalent to $\J_f=\tilde\J_f$.

\item\label{21b} If $X=V(f)$ is strongly Euler homogeneous, then so is $f$.
That is, strong Euler homogeneity is invariant under contact equivalence.

\item\label{21c} If $X\cong X'\times(\CC,0)$, then $X$ is strongly Euler homogeneous if and only if $X'$ is so (see \cite[Lem.~3.2]{GS06}).

\item\label{21d} For $X$ with trivial $\Iso X$ Euler homogeneity must be strong (see \cite[Thm.~1.(4')]{HM86}).
Note also that any $X$ with non-trivial $\Iso X$ is already Euler homogeneous.
Indeed, for $f$ independent of $x_n$, $X=V(\exp(x_n)f)$ and $\exp(x_n)f=\frac{\partial\exp(x_n)f}{\partial x_n}\in\J_{\exp(x_n)f}$.

\end{enumerate}
\end{rmk}


Euler homogeneity and Theorem~\ref{5} are linked by the following

\begin{rmk}\label{24}
Suppose that the hypersurface singularity $X=V(f)\subseteq Y$ is not smooth, that is, $\mathbf0\in\Sing X$.
Then there is a modified singular locus
\[
\Sing^*X:=V(TK(f)),\quad TK(f):=\mm_Y\J_f+\ideal f\unlhd\O_Y,
\]
defined by the tangent space $TK(f)$ at $f$ to the orbit of $f$ under the contact group.
The underlying reduced space germs of $\Sing X$ and $\Sing^*X$ agree.
Note that $TK(f)=\mm_Y\tilde\J_f$ if $X$ is strongly Euler homogeneous.

Thus, for strongly Euler homogeneous hypersurface singularities $X$ and $X'$, $\Sing X\cong\Sing X'$ implies $\Sing^*X\cong\Sing^*X'$.
By Gaffney and Hauser (see \cite[Part~I]{GH85}), this further implies that $X\cong X$ if $\dim X=\dim X'$.

However, Euler homogeneity is not a consequence of harmonicity.
In fact, isolated hypersurface singularities are trivially harmonic.
Correspondingly, Euler homogeneity does not suffice for the conclusion of Theorem~\ref{5} due to Remark~\ref{21}.\eqref{21d} and an example of Gaffney and Hauser (see \cite[\S4]{GH85}).
\end{rmk}


Strongly Euler homogeneous power series satisfy a Mather--Yau Theorem for right equivalence (see \cite[Thm.~9.1.10]{dJP00}).

\begin{thm}\label{4}
Let $f,f'\in\Cx$ be strongly Euler homogeneous power series.
Then the following statements are equivalent.
\begin{enumerate}[(a)]
\item\label{4a} The power series $f$ and $f'$ are right equivalent.
\item\label{4b} The power series $f$ and $f'$ are contact equivalent.
\item\label{4c} The $\CC$-algebras $\Cx/\J_f$ and $\Cx/\J_{f'}$ are isomorphic.
\end{enumerate}
\end{thm}


\begin{proof}
\eqref{4a} implies \eqref{4b} by Definition~\ref{31}, \eqref{4b} implies \eqref{4c} due to Remarks~\ref{32} and \ref{21}.\eqref{21a}.
It remains to show that \eqref{4c} implies \eqref{4a}.

An isomorphism $\Cx/\ideal\bx\J_f\cong\Cx/\ideal\bx\J_{f'}$ induced by \eqref{4c} is trivially one of algebras over $\CC\{f\}\cong\CC\{f'\}$ because the respective classes of $f$ and $f'$ are zero.
This implies \eqref{4a} (see \cite[Thm.~9.1.10]{dJP00}).
For the sake of self-containedness, we prove this latter implication:

By Remark~\ref{32}, we may assume that 
\[
\J_f=\J_{f'}=:\J
\]
and consider this as an equality of ideal sheaves on some common domain of convergence $\mathbf0\in U\subseteq Y$.
Consider the homotopy from $f$ towards $f'$
\[
H:=f+t\cdot(f'-f)\in\O_{U\times\CC}
\]
depending on a parameter $t\in\CC$ with its relative Jacobian ideal sheaf
\[
\J_H:=\ideal{\frac{\partial H}{\partial x_1},\dots,\frac{\partial H}{\partial x_n}}\unlhd\O_{U\times\CC}.
\]
By abuse of notation, we consider $\J\unlhd\O_{U\times\CC}$.
Then $\J_H\subseteq\J$ and 
\[
\frac{\partial f}{\partial x_i}\equiv t\cdot(\frac{\partial f'}{\partial x_i}-\frac{\partial f}{\partial x_i})\in t\J\mod\J_H\text{ for all }i=1,\dots,n.
\]
Thus, $\J/\J_H=t(\J/\J_H)$, and Nakayama's Lemma yields equality
\begin{equation}\label{12}
\J_H=\J\unlhd\O_{U\times\CC,(\mathbf0,0)}=\CC\{\bx,t\}.
\end{equation}
In other words, the support $Z$ of the quotient sheaf $\J/\J_H$ does not contain $(\mathbf0,0)$, and by symmetry not $(\mathbf0,1)$ either.
Due to coherence of the sheaf, $Z\subseteq U\times\CC$ is an analytic subset (see \cite[Cor.~6.2.9]{dJP00}), and $Z\cap(\set0\times\CC)$ identifies with a discrete subset $D\subseteq\CC\setminus\set{0,1}$ (see \cite[Thm.~3.1.10]{dJP00}).
For any $\tau\in\CC\setminus D$, strong Euler homogeneity of $f$ and $f'$ yields that
\[
\frac{\partial H}{\partial(t-\tau)}=f'-f\in\ideal\bx\J=\ideal\bx\J_H\unlhd\O_{U\times\CC,(\mathbf0,\tau)}=\CC\{\bx,t-\tau\}.
\]
By local triviality (see \cite[Cor.~9.1.6]{dJP00}), it follows that 
\[
f_\tau:=H(\bx,\tau)\in\O_Y=\CC\{\bx\}
\]
has locally constant right equivalence class for $\tau\in\CC\setminus D$.
Pick a continuous path $\gamma\colon[0,1]\to\CC\setminus D$ from $\gamma(0)=0$ to $\gamma(1)=1$.
By compactness, the right equivalence class of $f_\tau$ is then constant for $\tau\in\gamma([0,1])$.
In particular, $f=f_{\gamma(0)}$ and $f'=f_{\gamma(1)}$ are right equivalent, and hence, $X\cong X'$ as claimed.
\end{proof}

\section{Quasihomogeneous Jacobian algebras}

In this section, we deduce strong Euler homogeneity for a hypersurface singularity from a positive \emph{analytic} grading on the singular locus, generalizing a result of Xu and Yau (see \cite[Thm.~1.2]{XY96}).


The geometric meaning of strong Euler homogeneity is rather subtle in general.
It becomes more transparent in the following special case.


\begin{dfn}\label{11}
An \emph{Euler derivation} on a space germ $X$ is a $\CC$-linear derivation $\chi\colon\O_X\to\O_X$ of the form 
\begin{equation}\label{37}
\chi=\sum_{i=1}^nw_ix_i\frac\partial{\partial x_i},\quad \bw=w_1,\dots,w_n\in\QQ_{>0},
\end{equation}
where $\bx=x_1,\dots,x_n$ minimally generates $\mm_X$.
If it exists, then $X$ and all eigenvectors of $\chi$ are called \emph{quasihomogeneous}.
\end{dfn}


\begin{rmk}\label{41}
If $X=V(\I_{X/Y})\subseteq Y$ with $Y$ smooth and $\dim Y=\edim X$, then $\bx$ are coordinates on $Y$ and $\chi$ lifts to a $\CC$-linear derivation $\chi\colon\O_Y\to\O_Y$ with $\chi(\I_{X/Y})\subseteq\I_{X/Y}$ (see \cite[(2.1)]{SW73}).
Conversely, any such \emph{logarithmic Euler derivation} long $\I_{X/Y}$ induces an Euler derivation on $X$.
\end{rmk}


\begin{rmk}\label{56}
Any Thom--Sebastiani polynomial $f=\sum_{i=1}^k\bx^{\alpha_i}\in\O_Y$ is quasihomogeneous.
In fact, setting 
\[
w_i:=\begin{cases}
\frac1{\abs{\alpha_j}},& i\in[\alpha_j],\\
1,& i\not\in\bigsqcup_{j=1}^k[\alpha_j],
\end{cases}
\]
in \eqref{37} yields an Euler derivation $\chi$ on $Y$ such that $\chi(f)=f$.
\end{rmk}


\begin{rmk}\label{36}
Any quasihomogeneous $f\in\mm_Y$ is (strongly) Euler homogeneous since $\QQ_{>0}\cdot f\ni\chi(f)\in\mm_Y\J_f$ if $f\ne0$.
The converse holds true for isolated hypersurface singularities due to a result of K.~Saito (see \cite{Sai71}).
\end{rmk}


\begin{rmk}\label{35}
In more intrinsic terms, quasihomogeneity of $X$ means that $\mm_X$ is generated by eigenvectors $\bx=x_1,\dots,x_n$ of $\chi$ with eigenvalues $\bw\in\QQ_{>0}^n$.
By clearing denominators such that $\bw\in\ZZ_{>0}^n$, this becomes equivalent to a \emph{positive analytic grading} over $(\ZZ,+)$ in the sense of Scheja and Wiebe on the analytic algebra $(\O,\mm):=(\O_X,\mm_X)$, whose $k$th homogeneous part is the $k$-eigenspace of $\chi$, 
\[
\O_k=\ideal{\bx^\alpha\mid\ideal{\bw,\alpha}=k}_\CC,\quad
\ideal{\bw,\alpha}:=\sum_{i=1}^nw_i\alpha_i,
\]
spanned by monomials of $\bw$-weighted degree $k$ (see \cite[\S1-3]{SW73}).
This means that the vector spaces $\O_k\subseteq\O$, $k\in\ZZ$, induce a grading of $\O/\mm^\ell$, for all $\ell\in\ZZ_{\ge0}$, compatible with the canonical surjections $\O/\mm^\ell\to\O/\mm^{\ell}$, for all $\ell\ge\ell'$.
Note, however, that $\O\neq\bigoplus_{l\in\ZZ}\O_k$ in general.
\end{rmk}


Using Remark~\ref{35}, we can generalize a result of Xu and Yau (see \cite[Thm.~1.2]{XY96}) as follows.


\begin{thm}\label{3}
If a hypersurface singularity $X=V(f)$ has quasihomogeneous singular locus $\Sing X$, then $f$ is strongly Euler homogeneous.
\end{thm}

\begin{proof}
To reduce to the case where $\Iso X$ is trivial in \eqref{34}, apply Lemma~\ref{10} to $\Sing X$ and use Remark~\ref{21}.\eqref{21b}.
Then, by Remark~\ref{21}.\eqref{21d}, it suffices to to show that $f$ is Euler homogeneous.
This follows from the argument of Xu and Yau (see \cite[Proof of Thm.~1.2]{XY96}) using Lemma~\ref{15} and a positive \emph{analytic} grading on $\O_{\Sing X}=\O_Y/\tilde\J_f$ (see Remark~\ref{35}).
\end{proof}


\begin{lem}\label{10}
If a Cartesian product $X=X'\times Z$ of space germs is quasihomogeneous and $Z$ is smooth, then $X'$ is quasihomogeneous.
\end{lem}

\begin{proof}
Quasihomogeneity of $X$ yields an Euler derivation $\chi$ as in \eqref{37}.
By the Implicit Mapping Theorem (see \cite[Thm.~3.3.6]{dJP00}), reordering $\bx$ yields
\[
X'\cong X'':=V(x_1,\dots,x_k)\subseteq X,\quad k=\dim Z.
\]
The derivation induced by $\chi$ makes $X''$ and hence $X'$ quasihomogeneous.
\end{proof}


\begin{lem}\label{15}
Let $\O$ be an analytic algebra with maximal ideal $\mm$, and $\chi$ an Euler derivation as in \eqref{37}.
Then $\chi$ induces a $\CC$-linear automorphism on any $\chi$-invariant ideal $\I\subseteq\mm$.
\end{lem}

\begin{proof}
By Remark~\ref{41}, clearing denominators of $\bw$ and reordering variables, we may assume that $\O=\CC\{\bx\}$ in the situation of Remark~\ref{35} with
\begin{equation}\label{38}
w_1\le\dots\le w_n
\implies
w_1\abs{\alpha}\le\ideal{\bw,\alpha}\le w_n\abs{\alpha}.
\end{equation}
Expanding an element 
\[
f=\sum_{k\in\ZZ_{>0}}f_k\in\mm=\ideal\bx,\quad f_k\in\O_k, 
\]
in terms of $\bw$-weighted homogeneous parts, one finds a unique preimage 
\[
\int_\chi f:=\sum_{k\in\ZZ_{>0}}\frac{f_k}k,\quad\chi(\int_\chi f)=f.
\]
Writing $f=\sum_{\abs\alpha>0}f_\alpha\bx^\alpha$ in terms of monomials, then using \eqref{38} we obtain 
\[
\norm{\int_\chi f}_\bt=\sum_{\abs\alpha>0}\frac{\abs{f_\alpha}}{\ideal{\bw,\alpha}}\bt^\alpha\le\frac1{w_1}\sum_{\abs\alpha>0}\frac{\abs{f_\alpha}}{\abs{\alpha}}\bt^\alpha\le\frac1{w_1}\sum_{\abs\alpha>0}\abs{f_\alpha}\bt^\alpha=\frac{\norm{f}_\bt}{w_1}<\infty
\]
for all $\bt\in\RR_{+}^n$, and hence, $\int_\chi f\in\mm$ (see \cite[\S1.2, Satz~3', \S3.3]{GR71}).

If $\I\subseteq\mm$ is a $\chi$-invariant and hence $\bw$-weighted homogeneous ideal, then $\int_\chi$ leaves all homogeneous parts $\I_k$, $k\in\ZZ_{>0}$, and hence $\I$ itself invariant.
\end{proof}

\section{Monomial Jacobian ideals}\label{9}

In this section, we combine Proposition~\ref{1} and Theorems~\ref{4} and \ref{3} to prove our main Theorem~\ref{6}.


We consider the following extremal variant of quasihomogeneity given by a maximal number of linearly independent weight vectors.

\begin{dfn}\label{16}
We call an ideal of an analytic algebra $\O$ \emph{monomial} if it is generated by monomials in terms of some minimal generators of the maximal ideal $\mm\lhd\O$.
A space germ $X\subseteq Y$ is \emph{normal crossing} if its defining ideal $\I_{X/Y}\unlhd\O_Y$ is monomial.
Note that such space germs are quasihomogeneous.
\end{dfn}


\begin{rmk}
In more intrinsic terms, maximal quasihomogeneity of $X$ means that $\Aut_\CC\O_X$ contains an algebraic torus of dimension $\edim X$ (see Remark~\ref{39}) as a subgroup in the sense of Hauser and M\"uller (see \cite[\S1)]{HM89}).
Indeed, such a torus acts linearly in terms of suitable coordinates and lifts to any smooth space germ $Y\supseteq X$ with $\dim Y=\edim X$ (see \cite[Satz~6.i)]{HM89}).
The dimension condition is redundant because a general such embedding is isomorphic to $X\times Z\subseteq Y\times Z$ with $Z$ smooth.
The torus invariant defining ideal $\I_{X/Y}\unlhd\O_Y$ is then generated by monomials.
The converse implication holds trivially.
\end{rmk}


We are ready to prove our main result.


\begin{proof}[Proof of Theorem~\ref{6}]
Sufficiency is due to Remark~\ref{32}.
Suppose that $\tilde\J_f$ is monomial for some $0\ne f\in\ideal\bx$.
Then $\Sing X=V(\tilde\J_f)$ is quasihomogeneous, and hence, $f$ is strongly Euler homogeneous by Theorem~\ref{3}.
By Remark~\ref{13}, the support $F:=\supp(f)$ of $f$ defines a Jacobian semigroup ideal $J_F$.
Then $F'$ obtained from Proposition~\ref{1} is the set of exponents of a Thom--Sebastiani polynomial 
\[
f':=\sum_{\alpha\in F'}\bx^\alpha\in\CC\{\bx\},\quad\J_f=\J_{f'},
\]
which is strongly Euler homogeneous by Remarks~\ref{56} and \ref{36}.
Then $f$ and $f'$ are right equivalent due to Theorem~\ref{4}, proving the claim.
\end{proof}

\section{Logarithmic derivations and holonomicity}\label{50}

In this section, we describe the logarithmic derivations along hypersurface singularities defined by Thom--Sebastiani polynomials and show that they define a finite logarithmic stratification in the sense of K.~Saito (see \cite{Sai80}). 


\begin{dfn}\label{49}
The $\O_X$-module of \emph{logarithmic derivations} along the hypersurface singularity $X=V(f)\subseteq Y$ (see Remark~\ref{41} and \cite[(1.4)]{Sai80}),
\[
\Der(-\log X)\subseteq\Der_\CC\O_Y=:\Theta_Y,
\]
consists of all $\CC$-linear derivations $\delta\colon\O_Y\to\O_Y$ with $\delta(f)\subseteq\ideal{f}$.

The \emph{logarithmic stratification} of $X$ on $Y$ by smooth connected immersed submanifolds is characterized by the fact that, for all $y\in Y$, the tangent space at $y$ of the stratum containing $y$ is spanned by the evaluations of all elements of $\Der(-\log X)$ at $p$ (see \cite[(3.3)]{Sai80}).
If this stratification is finite, then $X$ is called \emph{holonomic} (see \cite[(3.8)]{Sai80}). 
\end{dfn}


\begin{rmk}\label{52}\
\begin{enumerate}[(a)]

\item\label{52a} Replacing $f$ by its squarefree part, that is, $X$ by the associated reduced space germ $X^{\red}$, does not affect logarithmic derivations and stratification.

\item\label{52c} The complement $Y\setminus X$ and the connected/irreducible components of $X\setminus\Sing(X^{\red})$ are (finitely many) logarithmic strata (see \cite[(3.4)\,iii)]{Sai80}).

\item\label{52b} The derivations annihilating $f$ form an $\O_X$-submodule 
\[
\Der(-\log X)\supseteq\Der(-\log f):=\ann_{\Theta_Y}(f)\cong\syz(\J_f),
\]
isomorphic to the syzygy module of $\J_f$.
Euler homogeneity of $X=V(f)$ yields a logarithmic vector field $\chi$ such that $\chi(f)=f$.
If suitably chosen, this yields a direct sum decomposition
\[
\Der(-\log X)=\O_Y\cdot\chi\oplus\Der(-\log f).
\]

\end{enumerate}
\end{rmk}


\begin{rmk}\label{61}
Consider a monomial (Thom--Sebastiani polynomial) $f=\bx^\alpha$ defining a normal crossing divisor $X=V(f)$.
With the Euler derivation $\chi$ from Remark~\ref{56} one verifies that
\[
\ann_{\Theta_Y}(f)=\ideal{x_i\frac\partial{\partial x_i}-\alpha_i\cdot\chi\mid i\in[\alpha]}+\ideal{\frac\partial{\partial x_i}\mid i\in[n]\setminus[\alpha]}.
\]
Since $X^{\red}=\bigcup_{i\in[\alpha]}V(x_i)$ it follows with Remark~\ref{52}.\eqref{52b} that
\[
\Der(-\log X)\vert_{X^{\red}}=\ann_{\Theta_Y}(f)\vert_{X^{\red}}
=\ideal{x_i\frac\partial{\partial x_i},\frac\partial{\partial x_j}\mid i\in[\alpha],j\in[n]\setminus[\alpha]}\vert_{X^{\red}}.
\]
This shows that $\ann_{\Theta_Y}(f)$ defines the logarithmic stratification of $X$ on $X$, and that the strata are relative complements of coordinate subspaces.
In particular, $X$ is holonomic by Remark~\ref{52}.\eqref{52c}.
\end{rmk}


\begin{prp}\label{51}
Any Thom--Sebastiani polynomial defines a holonomic hypersurface singularity.
\end{prp}

\begin{proof}
Let $f=\sum_{i=1}^k\bx^{\alpha_i}$ be a Thom--Sebastiani polynomial with support $F:=\set{\alpha_i\mid i\in[k]}$, defining a hypersurface singularity $X=V(f)\subseteq Y$.
In the case where $k=1$, the claim is due to Remark~\ref{61}.
The monomials of $f$ define (normal crossing) hypersurface singularities (see \eqref{55})
\[
X_i=V(\bx^{\alpha_i})\subseteq(\CC^{[\alpha_i]},\mathbf0)=:Y_i,\quad i=1,\dots,k,
\]
such that (see \eqref{57})
\[
X':=V(f)\subseteq\prod_{i=1}^k Y_i=:Y',\quad
X=X'\times Z\subseteq Y'\times Z=Y,\quad 
Z:=(\CC^{[n]\setminus[F]},\mathbf0),
\]
and all logarithmic strata of $X$ are products of strata of $X'$ with $Z$.
We may thus assume that $Z=\set{\mathbf0}$, that is, $[F]=[n]$.
Then
\[
\Sing X=\prod_{i=1}^k\Sing X_i.
\]
The Euler derivation from Remark~\ref{56} restricts to that in Remark~\ref{61} on each $Y_i$.
The syzygies $\syz(\J_f)$ are generated by all $\syz(\J_{\bx^{\alpha_i}})$ and the Koszul relations.
These latter vanish on $\Sing X$.
By Remarks~\ref{52}.\eqref{52b} and \ref{61}, it follows that the logarithmic strata of $X$ in $\Sing X$ are products of finitely many strata of the $X_i$ in $\Sing X_i$.
Thus, $X$ is holonomic by Remark~\ref{52}.\eqref{52c}.
\end{proof}

\printbibliography
\end{document}